\documentclass[11pt]{amsart}
\usepackage[all]{xy}
\usepackage{amscd}
\usepackage{amsfonts}
\usepackage{amsmath}
\usepackage{amssymb}
\usepackage{amsthm}
\usepackage{bm}
\usepackage{CJK}
\usepackage{dsfont}
\usepackage{fancyhdr}
\usepackage{graphics}
\usepackage{graphicx}
\usepackage{latexsym}
\usepackage{mathrsfs}
\usepackage{parskip}
\usepackage{slashed}
\usepackage{titletoc}
\usepackage[usenames]{color}
\allowdisplaybreaks

\begin{CJK*}{GBK}{kai}
\newtheorem{thm}{Theorem}[section]
\end{CJK*}

\newtheorem{lem}[thm]{Lemma}
\newtheorem{prop}[thm]{Proposition}

\numberwithin{equation}{section}

\newcommand{\mn}{\sqrt{-1}}
\newcommand{\de}{\partial}
\newcommand{\dbar}{\overline{\partial}}

\newcommand{\md}{\mathrm{d}}
\newcommand{\dist}{\mathrm{dist}}
\newcommand{\dmu }{\mathrm{d}\mu }

\newcommand{\mS }{\mathcal{S}}
\newcommand{\dmut}{\mathrm{d}\mu(t)}

\newcommand{\ov}{\overline}
\newcommand{\lf}{\left}
\newcommand{\rt}{\right}
\newcommand{\intm }{\int_{M} }
\newcommand{\n}{\nabla}
\newcommand{\C}{ \mathbb{C}}
\newcommand{\ld}{\lambda}
\newcommand{\bag}{\begin{aligned}}
\newcommand{\eag}{\end{aligned}}
\newcommand{\R}{ \mathbb{R}}
\newcommand{\be}{\begin{equation}}
\newcommand{\ee}{\end{equation}}
\newcommand{\ti}{\tilde}
\newcommand{\ba}{\begin{eqnarray}}
\newcommand{\ea}{\end{eqnarray}}
\newcommand{\no}{\nonumber}
\newcommand{\al}{\alpha}

\begin{document}
\begin{CJK}{GBK}{song}
\title{Isoperimetric Inequality along the Twisted K\"{a}hler-Ricci Flow}
\makeatletter
\let\uppercasenonmath\@gobble% disables title uppercase
\let\MakeUppercase\relax% disables author uppercase
\let\scshape\relax% disables section smallcaps
\makeatother
\author{Shouwen Fang$^1$, Tao Zheng$^2$}
\address{$^1$School of Mathematical Science, Yangzhou University,
Yangzhou, Jiangsu 225002, P. R. China }
\email{shwfang@163.com}
\address{$^2$Institut Fourier, Universit\'{e} Grenoble Alpes, 100 rue des maths,
Gi\`{e}res 38610, France}
\email{zhengtao08@amss.ac.cn}
%\date{ }
\thanks{$^1$ Supported by National Natural Science Foundation of China grant Nos.11401514 and 11371310.}
\thanks {$^2$ Supported by National Natural Science Foundation of China grant No. 11401023, and the author's post-doc is supported by the European Research Council (ERC) grant No. 670846 (ALKAGE)}
\begin{abstract}
We prove a uniform isoperimetric inequality for all time along the twisted K\"{a}hler-Ricci flow on Fano manifolds.
\end{abstract}
\maketitle
\section{Introduction}
The classical \emph{isoperimetric inequality} states that for Borel set $\Omega\in \mathbb{R}^n (n\geq2)$ with finite Lebesgue measure $|\Omega|$, the ball with the same measure has a lower perimeter, that is,
\begin{equation}
\label{ciso}
P(\Omega)\geq n\omega_{n}^{\frac{1}{n}}|\Omega|^{\frac{n-1}{n}},
\end{equation}
where $P(\Omega)$ is the distributional perimeter of $\Omega$ which coincides with the classical $n-1$-dimensional area of $\de \Omega$ if $\Omega$ has smooth boundary and $\omega_{n}$ is the volume of unit ball in $\mathbb{R}^n$. It is also well-known that equality holds in  \eqref{ciso} if and only if $\Omega$ is a ball $B$ in $\mathbb{R}^n$. De Giorgi \cite{degior1}(see also \cite{degior2} for English version) proved (\eqref{ciso}) for the first time in the general framework of sets with finite perimeter. One can find various kinds of proofs and different formulations of the isoperimetric inequality  in \cite{bla,bur,chavel,fusco,oss,talen} and references therein.
%Then there is a long and complex history of the various kinds of proofs and different formulations of the isoperimetric inequality (see \cite{bla,bur,chavel,fusco,oss,talen} and references therein).

In the case of geometric flows, Hamilton \cite{ha} obtained an isoperimetric estimate for the Ricci flow on the two sphere.
For complex $2$-dimensional K\"{a}hler-Ricci flow, Chen and Wang \cite{chenwango} proved that the isoperimetric constant for $(M,\,g(t))$ is bounded from below  by a uniform constant. Here $g(t)$ is the solution of the K\"{a}hler-Ricci flow (see \eqref{tkrf} with $\theta_{i\overline{j}}\equiv0$).  Later, Tian and Zhang \cite{tianzhang} proved that, for all complex $n$-dimensional K\"{a}hler-Ricci flow on Fano manifolds, the isoperimetric constant for $(M,\,g(t))$ is also bounded from below  by a uniform constant.

In this paper, we obtain a uniform estimate of lower bound on isoperimetric constant  along the twisted  K\"{a}hler-Ricci flow on Fano manifolds. To be precise, we need some notations and definitions.
Let $M$ be a real $n(=2m)$ dimensional Fano manifold with K\"{a}hler form $\omega_0$ associated to the K\"{a}hler metric  $g_0$. We consider the twisted K\"{a}hler-Ricci flow (see \cite{cabor,liu,zhangzhang} and the references therein)
\begin{equation}
\label{tkrf}
\left\{
\begin{array}{rl}
\frac{\partial}{\partial t} g_{i\ov{j}}(x,t)=&-R_{i\ov{j}}(x,t)+\theta_{i\overline{j}}(x)+g_{i\overline{j}}(x,t),\\
g_{i\overline{j}}(x,0)=&(g_0)_{i\overline{j}}(x ),
\end{array}
\right.
\end{equation}
where $\theta $ is a closed  semi-positive $(1,1)$ form and
$$
[2\pi c_1(M)]=[\omega(x,t)+\theta].
$$
Here $\omega(x,t)=\mn g_{i\overline{j}}(x,t)\md z^i\wedge\md \overline{z}^j$ associated to  the K\"{a}hler metric $g(x,t)$.   For convenience, we denote
$$
\mS_{i\overline{j}}(x,t)=R_{i\ov{j}}(x,t)-\theta_{i\ov{j}}(x)
$$
and
$$
S(x,\,t)=2\sum\limits_{i,j=1}^m g^{\ov{j}i}(x,\,t)\mS_{i\ov{j}}(x,\,t).
$$
We know that
\begin{prop}\label{prop1}
For the twisted K\"{a}hler-Ricci flow \eqref{tkrf} on Fano manifolds, there  exist uniform positive constants $C$, $\kappa$ and $C_S$ such that
\begin{enumerate}
\item[(a)] $|S(x,t)|\leq C$,
\item[(b)] $|\mathrm{diam}(M,\,g(t))|\leq C$,
\item[(c)] $\|h\|_{C^{1}}\leq C$, where from $\de\ov{\de}$-lemma, $h\in C^{\infty}(M,\,\mathbb{R})$ satisfies
           \begin{equation}
           \label{kpotential}
           \mS_{i\ov{j}}-g_{i\ov{j}}=\de_{i}\de_{\ov{j}}h,
           \end{equation}
\item[(d)] $\mathrm{Vol}_{g(t)}(B(x,r,t))\geq \kappa r^n$,  for any $t>0$ and $r\in (0,\,\mathrm{diam}(M,\,g(t)))$,
\item[(e)] $\mathrm{Vol}_{g(t)}(B(x,r,t))\leq \kappa^{-1} r^n$, for any $t>0$ and $r>0$,
\item[(f)] for any $f \in W^{1,2}(M)$,
           $$
           \lf(\intm |f|^{\frac{2n}{n-2}}\dmut\rt)^{\frac{n-2}{n}}\leq C_S\lf(\intm  \big[|\n f|_{g(t)}^2+f^2\big]\dmut\rt).
           $$
\end{enumerate}
\end{prop}
Here and henceforth, by a uniform constant we mean  a constant depending only on the initial data on $M$.
Items (a)-(d) in Proposition \ref{prop1} can be founded in \cite{cabor,liu} and items (e)-(f) in Proposition \ref{prop1} can be founded in \cite{fangzheng}.
Since the volume of $(M,\,g(t))$ is a constant, from item (e) in Proposition \ref{prop1}, there exists a uniform constant $\beta>0$ such that
$$
\mathrm{diam}(M,\,g(t))>\beta.
$$
In the case of K\"{a}hler-Ricci flow, i.e., $\theta_{i\ov{j}}\equiv0$,  Items (a)-(d) in Proposition \ref{prop1} is due to Perelman (See \cite{sesumtian}). Item (e)  in Proposition \ref{prop1} belongs to \cite{chenwang,zhangqi3} and Item (f)  was established by \cite{ye1,zhangqi1} (see also \cite{zhangqi2}).

As a consequence of Proposition \ref{prop1}, we can deduce our main theorem as follows.
\begin{thm}\label{thmiso}
For the twisted K\"{a}hler-Ricci flow (\ref{tkrf}) on Fano manifolds, for any $u \in C^{\infty}(M,\,\mathbb{R})$, there holds
the isoperimetric inequality
$$
\|f(x)-f_M\|_{L^{\frac{n}{n-1}}(M)}\leq C_I\|\n f\|_{L^1(M)},\quad f\in \,C^{\infty}(M,\,\mathbb{R}),
$$
where $C_I>0$ is a uniform constant.
\end{thm}
%\begin{thm}\label{thmiso}
%For the twisted K\"{a}hler-Ricci flow (\ref{tkrf}) on Fano manifolds, for any $u \in C^{\infty}(M,\,\mathbb{R})$, there holds
%\begin{equation}\label{sobtkrf}
%\lf(\intm u^{\frac{ n}{n-1}}\dmut\rt)^{\frac{n-1}{n}}\leq S_1 \intm   |\n u|_{g(t)} \dmut
%+\frac{C}{\lf[\mathrm{Vol}_{g(t)}(M)\rt]^{\frac{1}{n}}}\intm|u| \dmut,
%\end{equation}
%where $S_1>0$ is a uniform constant and $C$ is a positive numerical constant.
%The Neumann-Sobolev inequality  \eqref{sobtkrf}  implies isoperimetric inequality
%$$
%\|f(x)-f_M\|_{L^{\frac{n}{n-1}}(M)}\leq C_I\|\n f\|_{L^1(M)},\quad f\in \,C^{\infty}(M,\,\mathbb{R}),
%$$
%where $C_I>0$ is a uniform constant.
%\end{thm}
From Theorem \ref{thmiso}, we can get a uniform lower bound for the isoperimetric constant in $(M,\,g(t))$ as follows.
There holds
$$
I(M,\,g(t)):=\inf_{V\subset M}\frac{\mathrm{Area}_{g(t)}(\partial V)}{\lf[\min\lf\{\mathrm{Vol}_{g(t)}(V),\,\mathrm{Vol}_{g(t)}(M-V)\rt\}\rt]^{\frac{n-1}{n}}}\geq \delta,
$$
where $V$ is a subdomain of $M$ such that $\de V$ is an $n-1$ dimensional submanifold of $M$, and $\delta$ is a positive constant depending only on initial metric $g_0$. A proof can be found in \cite[Section 5.1]{cln}.

In the case of K\"{a}hler-Ricci flow, the theorem and the isoperimetric inequality above were obtained by Tian and Zhang \cite{tianzhang}.

\noindent {\bf Acknowledgements}
This work was carried out while the authors were visiting Mathematics Department at Northwestern University. We would like to thank Professor Valentino Tosatti and Professor Ben Weinkove for hospitality and helpful discussions. We also thank  Professor Qi S. Zhang for offering us more details about heat kernel and Wenshuai Jiang for some helpful conversations. The authors are also grateful
to the anonymous referees and the editor for their careful reading and helpful suggestions which
greatly improved the paper.
\section{The uniform Poincar\'{e} inequality along the twisted K\"{a}hler-Ricci flow}
In this section, we will use the estimate of the first non-zero eigenvalue of self-adjoint elliptic operator
\begin{equation}
\label{twistlap}
\Delta_{h}f=\ov{\de}^{\ast}\dbar f-\sum\limits_{i=1}^{m}\n^i f\n_i h,\quad \forall\; f\in C^{\infty}(M,\,\C),
\end{equation}
where $h$ is defined in (\ref{kpotential}) and $\ov{\de}^{\ast}$ is the conjugated operator of $\ov{\de}$ with respect to $g(t)$,
in order to get the uniform Poincar\'{e} inequality  along the twisted K\"{a}hler-Ricci flow \eqref{tkrf}.
\begin{lem}
Let $g(t)$ be the solution to the twisted K\"{a}hler-Ricci flow \eqref{tkrf} on $M\times[0,\,\infty)$. Then for any $f\in C^{\infty}(M,\,\R)$, there holds the uniform Poincar\'{e} inequality
\begin{equation}
\label{poincareinequality}
\intm  \lf(f-\frac{1}{\mathrm{Vol}_{g(t)}(M)}\intm f \dmut\rt) ^2 \dmut\leq C\intm |\n f|_{g(t)}^2\dmut,
\end{equation}
with a uniform constant $C$.
\end{lem}
\begin{proof}
Let $\lambda_1$ be the first non-zero eigenvalue of $\Delta_h$ defined as in \eqref{twistlap}.
Then for any $\varphi\in C^{\infty}(M,\,\R)$ with $\intm \varphi e^{h}\dmut=0$, from \cite[Theorem 2.3.4]{futaki}, we get
$$
\bag
\ld_1 \intm |\dbar \varphi|_{g(t)}^2 e^h \dmut
=&\intm \sum\limits_{i,j=1}^{m}\lf(R_{j\ov{i}}(g(t))\n^{j}\varphi\n^{\ov{i}}\ov{\varphi}-\n_{\ov{i}}\n_{j}h \n^{j}\varphi\n^{\ov{i}}\ov{\varphi}+\n^{i}\n^{j}\varphi\n_{i}\n_{j}\overline{\varphi}\rt)e^h\dmut\\
=&\intm \lf(|\dbar \varphi|_{g(t)}^2 +\sum\limits_{i,j=1}^{m}\theta_{j\ov{i}}\n^{j}\varphi\n^{\ov{i}}\ov{\varphi} +\sum\limits_{i,j=1}^{m}\n^{i}\n^{j}\varphi\n_{i}\n_{j}\overline{\varphi}\rt)e^h\dmut,
\eag
$$
where for second equality we use \eqref{kpotential}. This implies $ \ld_1\geq 1. $
Hence, for any $\varphi\in C^{\infty}(M,\,\R)$ with $\intm \varphi e^{h}\dmut=0$, it follows
\be\label{pqine}
\intm \varphi^2e^h\dmut\leq \intm |\dbar \varphi|_{g(t)}^2e^h\dmut.
\ee
Now for any $\psi\in C^{\infty}(M,\,\R)$ with $\intm \psi \dmut=0$, from (\ref{pqine}), it follows
\be\label{pine1}
\intm \lf(\psi-\frac{\intm \psi e^h\dmut}{\intm   e^h\dmut}\rt)^2e^h\dmut\leq \intm |\dbar \psi|_{g(t)}^2e^h\dmut.
\ee
By Item (c) of Proposition \ref{prop1} and (\ref{pine1}), there exists a uniform constant $C$ such that
$$
\bag
C\intm |\dbar \psi|_{g(t)}^2\dmut
\geq& \intm |\dbar \psi|_{g(t)}^2e^h\dmut\\
\geq& \intm \lf(\psi-\frac{\intm \psi e^h\dmut}{\intm   e^h\dmut}\rt)^2e^h\dmut\\
\geq& C^{-1}\intm \lf(\psi-\frac{\intm \psi e^h\dmut}{\intm  e^h\dmut}\rt)^2 \dmut\\
=& C^{-1}\lf(\intm  \psi ^2 \dmut+\lf[\frac{\intm \psi e^h\dmut}{\intm   e^h\dmut}\rt]^2\mathrm{Vol}_{g(t)}(M)\rt)\\
\geq&C^{-1}\intm  \psi ^2 \dmut,
\eag
$$
which implies the Poincar\'{e} inequality \eqref{poincareinequality}.
\end{proof}
Letting $\theta_{i\overline{j}}=0$, we get the uniform Poincar\'e inequality along the K\"ahler-Ricci flow which was obtained by Tian and Zhang \cite{tianzhang} using the estimates of heat kernel.
\section{The uniform gradient estimates along the twisted K\"ahler-Ricci flow}
By Proposition \ref{prop1} and the uniform Poincar\'e inequality \eqref{poincareinequality}, we can prove the isoperimetric inequality.
Since the argument is for every fixed time $t$, for convenience, in the following we will omit the time $t$ in all the quantities such as distance, diameter, norm of the gradient, Laplacian, volume element, etc. Hence our arguments are the same as the ones in the Riemannian case formally. The following lemma is essentially due to Tian and Zheng \cite{tianzhang}.
\begin{lem}
Let $g(t)$ be the solution to the twisted K\"{a}hler-Ricci flow \eqref{tkrf} on $M\times[0,\,\infty)$.
Then
\begin{enumerate}
\item[(a)]for smooth harmonic function $u$ in $B(x_0,\,r)$ with $r\leq \mathrm{diam}(M)$, we have
\be\label{bfo7}
\sup_{x\in B(x_0,\,\frac{r}{2})} |\n u(x)|\leq C\frac{1}{r}\lf(\frac{1}{r^n}\int_{B(x_0,\,r)}u^2\dmu\rt)^{\frac{1}{2}},
\ee
where $C$ is a uniform constant.
\item[(b)] for a non-negative smooth funtion $u$ on $M\times [0,\,\infty)$ satisfying
\be\label{heatequation}
\de_s u- \Delta u= 0,
\ee
we have
\be\label{bfo21}
|\n u(x,\,s)|\leq \frac{C}{\sqrt{\eta \tilde s}}\lf(\frac{1}{(\eta \tilde s)^{\frac{n+2}{2}}}\int_{s-\eta \tilde s}^s\int_{B(x,\,\sqrt{\eta \tilde s})}u^2\dmu\md v\rt)^{\frac{1}{2}}
\ee
where $0<\eta\leq1$ is a parameter, $\tilde s=\min\{s,\,\mathrm{diam}(M)^2\}$ and $C$ is a unform constant.
\end{enumerate}
\end{lem}
We also need the following lemma whose proof is the same as the ones in Grigor'yan \cite{grigor} and Saloff-Coste \cite{saloff}.
\begin{lem}
Let $g(t)$ be the solution to the twisted K\"{a}hler-Ricci flow \eqref{tkrf} on $M\times[0,\,\infty)$.
Assume that $u$ is a smooth harmonic function in $B(x_0,\,r)$ where $x_0\in M,\,r\in (0,\,\mathrm{diam}(M)]$.
There exists a positive constant $C_1=C_1(C_S,\,p)$ such that
\be\label{hmoser}
\sup_{x\in B(x_0,\,\sigma r)}|u(x)|\leq C_{1}\lf(1+\frac{1}{(\mu-\sigma)^2 r^2}\rt)^{\frac{n}{2p}}\lf( \int_{B(x_0,\,\mu r)}u^p\dmu\rt)^{\frac{1}{p}},
\ee
where $0<p<+\infty$ and $0<\sigma<\mu\leq 1$.

In addition, denote $\tilde s=\min\{s,\,\mathrm{diam}(M)^2\}$ and
$$
Q_{\delta}=B\lf(x_0,\,\delta \sqrt{\eta \ti s}\rt)\times\lf[s -\delta  \eta \ti s,\,s\rt]
,\;
Q=B\lf(x_0,\,\sqrt{\eta \ti s}\rt)\times\lf[s-\eta \ti s,\,s\rt].
$$
If  $u$ is a solution of  heat equation (\ref{heatequation})
in the space time cube
$
M\times [0,\, +\infty),
$
then given $0<\delta<1$, there holds
\be\label{pmoser}
\sup_{(x,v)\in Q_{\delta}}|u(x,\,v)|^p\leq C_2(1-\delta)^{-(n+2)}(\eta\ti s)^{-\frac{n+2}{2}}(1+\eta\ti s)^{\frac{n+2}{2}} \int_{Q}[u(\cdot,\nu)]^p\dmu\md \nu,
\ee
where $C_2$ depends on $C_S$ and $p$ with $0<p<+\infty$ and $0<\eta\leq 1$ is a parameter.
\end{lem}
Denote by $H(x,\,y,\,s)$ the heat kernel of heat equation (\ref{heatequation}). We have the following estimates for the heat kernel.
\begin{lem}
Let $g(t)$ be the solution to the twisted K\"{a}hler-Ricci flow \eqref{tkrf} on $M\times[0,\,\infty)$. Then we have
\be\label{heatupperbound}
H(x,\,y,\,s) \leq C_1\lf(1+\frac{[\dist(x,\,y)]^2}{s}\rt)^{\frac{n}{2}}  \ti s ^{-\frac{n}{2}} e^{-\frac{[\dist(x,\,y)]^2}{4s}}
\ee
and
\be\label{heatgradient}
|\n_{x} H(x,\,y,\,s)|\leq C_2\lf(1+\frac{[\dist(x,\,y)]^2}{s}\rt)^{\frac{n+1}{2}}\tilde{s}^{-\frac{n+1}{2}}e^{-\frac{[\dist(x,\,y)]^2}{4s}},
\ee
where $\ti s=\min\{s,\,\mathrm{diam}(M)^2\}$, $C_1$ is a uniform constant.
\end{lem}
\begin{proof}
Fix $\ld \in\R$ and a bounded function $\psi$ satisfying $|\n\psi|\leq 1$. For any nice complex function $f$, set
$f_{z}(y)=e^{\ld\psi(y)}\lf[e^{z\Delta}(e^{-\ld\psi}f)\rt](y)$ for $z=se^{\mn\theta}>0,\,s>0,\,|\theta|\leq \frac{1}{2} \varepsilon$, where $0<\varepsilon\ll1$ is a small parameter. Saloff-Coste  \cite{saloff1} proved
\be\label{yongde}
\|f_z\|_{2}^2\leq e^{2\ld^2(1+\varepsilon)s}\|f\|_{2}^2.
\ee
Introduce function
$$
u(y,\,s)=e^{-\ld \psi(y)}f_s(y)=\lf[e^{s\Delta}(e^{-\ld\psi}f)\rt](y),
$$
where $f\in L^{2}(M)$ is a real function.
The function $u$ satisfies the heat equation (\ref{heatequation}).

Thus, from (\ref{pmoser}), we have
\ba\label{pmoserr1}
|u(x,\,s)|^2&\leq&C(1+\eta \ti s)^{\frac{n+2}{2}}(\eta \ti s)^{-\frac{n+2}{2}}
\int_{s-\eta \ti s}^{s}\int_{B\lf(x,\;\sqrt{\eta \ti s}\rt)}|u(z,\,\nu)|^2\dmu\md \nu.
\ea
Here for later use,  we take $0<\eta\ll1$ determined later.

Multiplying both sides of (\ref{pmoserr1}) by $e^{2\ld\psi(x)}$, from (\ref{yongde}), we get
\ba\label{star}
e^{2\ld\psi(x)}|u(x,\,s)|^2&\leq&C(1+\eta \ti s)^{\frac{n+2}{2}}(\eta \ti s)^{-\frac{n}{2}}e^{2|\ld|\sqrt{\eta \ti s}+2\ld^2(1+\varepsilon)s}\|f\|_{L^2(M)}^2.
\ea
Take cut-off function $\varphi(z)$ such that
$\varphi(z)=1$ on $B(y,\,\sqrt{\eta \ti s})$ and $\varphi(z)=0$ on $M-B(y,\,(1+\epsilon)\sqrt{\eta \ti s})$, where $0<\epsilon\ll1$ small enough.
Choosing
$$
f(z)=\frac{\varphi(z)H(x,z,s)}{\|\varphi(z)H(x,z,s)\|_{L^2(M)}},
$$
we obtain
\be\label{pmoser1}
\bag
&e^{\ld(\psi(x)-\psi(y))}\intm e^{2|\ld|\sqrt{\eta \ti s}+\ld(\psi(y)-\psi(z))}H (x,z,s)\frac{\varphi(z)H(x,z,s)}{\|\varphi(z)H(x,z,s)\|_{L^2(M)}}
\md \mu(z)\\
\geq &e^{\ld(\psi(x)-\psi(y))}\int_{B(y,\,(1+\epsilon)\sqrt{\eta \ti s})}  H (x,z,s)\frac{\varphi(z)H(x,z,s)}{\|\varphi(z)H(x,z,s)\|_{L^2(M)}}
\md \mu(z)\\
\geq& e^{\ld(\psi(x)-\psi(y))}\|\varphi(z)H(x,z,s)\|_{L^2(B(y,\,(1+\epsilon)\sqrt{\eta \ti s}))}\\
\geq &e^{\ld(\psi(x)-\psi(y))}\|H(x,z,s)\|_{L^2(B(y,\, \sqrt{\eta \ti s}))}.
\eag
\ee
Considering $H(x,z,s)$ as a function of $z$, from (\ref{star})  and  (\ref{pmoser1}), we deduce
\be\label{heatl2e}
\|H(x,z,s)\|_{L^2(B(y,\, \sqrt{\eta \ti s}))}\leq C(1+\eta \ti s)^{\frac{n+2}{4}}(\eta \ti s)^{-\frac{n}{4}}e^{3|\ld| \sqrt{\eta \ti s} + \ld^2(1+\varepsilon)s-\ld(\psi(x)-\psi(y))}
\ee
By using (\ref{pmoser}), we have
\ba
|H(x,\,y,\,s)|^2&\leq&C(1+\eta \ti s)^{\frac{n+2}{2}}(\eta \ti s)^{-\frac{n+2}{2}}\no\\
&&\int_{s-\eta \ti s}^{s}\int_{B(y,\;\sqrt{\eta \ti s})}|H(x,\,z,\,s)|^2\dmu(z)\md \nu\no\\
&\leq&C(1+\eta \ti s)^{\frac{n+2}{2}}(1+\eta \ti s)^{\frac{n+2}{2}}(\eta \ti s)^{-\frac{n}{2}}(\eta \ti s)^{-\frac{n}{2}}\no\\
&&e^{6|\ld|\sqrt{\eta \ti s}+ 2\ld^2(1+\varepsilon)s-2\ld(\psi(x)-\psi(y))}\no,
\ea
i.e.,
\ba
H(x,\,y,\,s)&\leq&C(1+\eta \ti s)^{\frac{n+2}{4}}(1+\eta \ti s)^{\frac{n+2}{4}}(\eta \ti s)^{-\frac{n}{4}}(\eta \ti s)^{-\frac{n}{4}}\no\\
&&e^{3|\ld|\sqrt{\eta \ti s}+  \ld^2(1+\varepsilon)s- \ld(\psi(x)-\psi(y))}\no.
\ea
Combining (\ref{bfo21}) and (\ref{heatl2e}), we can derive
\be\label{gheat}
\bag
|\n_{x} H(x,\,y,\,s)|\leq
&\frac{C}{\sqrt{\eta \ti s}}\lf(\frac{1}{(\eta \ti s)^{\frac{n+2}{2}}}\int_{s-\eta \ti s}^t\int_{B(x,\,\sqrt{\eta \ti s})}[ H(y,\,z,\,\nu)]^2\dmu(z)\md \nu\rt)^{\frac{1}{2}}\\
\leq&C(1+\eta \ti s )^{\frac{n+2}{4}}(\eta \ti s)^{-\frac{n+1}{2}}e^{3|\ld|\sqrt{\eta \ti s} + \ld^2(1+\varepsilon)s-\ld(\psi(x)-\psi(y))}.
\eag
\ee
Finally, taking $\psi$ such that $\psi(x)-\psi(y)=\dist(x,\,y)$ and
$$
\ld=\frac{\dist(x,\,y)}{2(1+ \varepsilon)s},
\quad
\eta=\frac{1}{10\lf(1+\frac{[\dist(x,\,y)]^2}{s}\rt)},
$$
we can deduce (\ref{heatupperbound}) and (\ref{heatgradient}).
\end{proof}
Denote by
$$
G_{0}(x,\,y):=\int_{0}^{+\infty}\lf(H(x,\,y,\,t)-\frac{1}{\mathrm{Vol}(M)}\rt)\md t
$$ the Green function of Laplacian.
\begin{lem}
\label{lemg}
Let $g(t)$ be the solution to the twisted K\"{a}hler-Ricci flow \eqref{tkrf} on $M\times[0,\,\infty)$.
Then there hold
\begin{enumerate}
\item[(a)] $|G_{0}(x,\,y)|\leq \frac{C}{[\dist(x,\,y)]^{n-2}},\quad x,\,y\in M$,
\item[(b)] $|\n_{x}G_{0}(x,\,y)|\leq \frac{C}{[\dist(x,\,y)]^{n-1}},\quad x,\,y\in M$,
\end{enumerate}
where $C$ is a uniform constant.
\end{lem}
\begin{proof}
For any $f_{0}\in C^{\infty}(M,\,\R)$ with $\intm f_0\dmu=0$,
$$
f(x,\,s)=\intm\lf(H(x,z,s)-\frac{1}{\mathrm{Vol}(M)}\rt)f_0(z)\dmu(z)
$$
is the solution of heat equation (\ref{heatequation}) satisfying
$$
\intm f(x,s)\dmu(x)=0,\quad f(x,\,0)=f_{0}(x).
$$
From the Poincar\'{e} inequality (\ref{poincareinequality}), we get
$$
\de_s\intm [f(x,s)]^2\dmu(x)=-2\intm|\n f(x,\,s)|^2\dmu(x)\leq -C^{-1}\intm [f(x,s)]^2\dmu(x),
$$
which implies
\be\label{pso1}
\intm [f(x,s)]^2\dmu(x)\leq e^{-\frac{t}{C}}\intm [f_0(x)]^2\dmu(x), \quad s>0.
\ee
Combining (\ref{pmoser}) and  (\ref{pso1}), for $t\geq 10\beta^2$, we have
\ba
[f(x,s)]^2
&\leq&\frac{C_1}{\beta^{n+2}}\int_{s-\beta^2}^{s}\intm [f(z,\,\nu)]^2\dmu(z)\md \nu\no\\
&\leq&C_1 e^{-\frac{s}{C}}\intm [f_0(x)]^2\dmu(x)\no,
\ea
i.e.,
\be\label{pso2}
\lf\{\intm\lf(H(x,z,s)-\frac{1}{\mathrm{Vol}(M)}\rt)f_0(z)\dmu(z)\rt\}^2\leq C_1 e^{-\frac{t}{C}}\intm [f_0(x)]^2\dmu(x),
\ee
where $C_1$ is a uniform constant.

For $s\geq 10\beta^2$ and $x$ fixed, taking $f_0(z)=H(x,z,s)-\frac{1}{\mathrm{Vol}(M)}$, from (\ref{pso2}), we can deduce
\be\label{pso3}
\intm\lf(H(x,z,s)-\frac{1}{\mathrm{Vol}(M)}\rt)^2\dmu(z)\leq C_1 e^{-\frac{s}{C}},\quad s\geq 10\beta^2.
\ee
For $x$ fixed, the function $H(x,z,s)-\frac{1}{\mathrm{Vol}(M)}$ of $z$ is also the solution of heat equation (\ref{heatequation}). Thus,
from  (\ref{pmoser}) and  (\ref{pso3}), for $s\geq 10\beta^2$, we can derive
$$
\bag
\lf(H(x,y,s)-\frac{1}{\mathrm{Vol}(M)}\rt)^2
\leq&\frac{C_1}{\beta^{n+2}}\int_{s-\beta^2}^{s}\intm \lf(H(x,z,s)-\frac{1}{\mathrm{Vol}(M)}\rt)^2\dmu(z)\md \nu\\
\leq&  C_1 e^{-\frac{s}{C}},
\eag
$$
i.e.,
\be\label{pso4}
\lf|H(x,y,s)-\frac{1}{\mathrm{Vol}(M)}\rt|\leq C_2 e^{-\frac{s}{C}},
\ee
where $C_2$ and $C$ are uniform constants.

Noticing
$$
G_0(x,\,y)=\int_{0}^{10\beta^2}\lf(H(x,y,s)-\frac{1}{\mathrm{Vol}(M)}\rt)\md s+\int_{10\beta^2}^{\infty}\lf(H(x,y,s)-\frac{1}{\mathrm{Vol}(M)}\rt)\md s,
$$
by (\ref{heatupperbound}) and  (\ref{pso4}), we can prove Item (a).

Item (b) follows from Item (a) and (\ref{bfo7}).
\end{proof}

\section{Proof of Theorem \ref{thmiso}}
In this section, for convenience, denote by $|\Omega|$ the volume of the set $\Omega$ with respect to the metric $g(t)$.
For any $f\in C^{\infty}(M,\,\R)$, the Hardy-Littlewood maximal function $Mf$ is defined by
$$
(Mf)(x):=\sup_{r>0}\frac{1}{\lf|B(x,\,r)\rt|}\int_{B(x,\,r)}|f(\xi)|\dmu(\xi),
$$
and we also define
$$
(I_{\al}f)(x):=\intm |f(\zeta)|\frac{[\dist(x,\,\zeta)]^{\al}}{\lf|B(x,\dist(x,\zeta))\rt|}\dmu(\zeta),\quad 0<\al<n.
$$
First, we give a lemma about the Hardy-Littlewood maximal function as follows.
\begin{lem}
\label{mfunction}
Let $g(t)$ be the solution to the twisted K\"{a}hler-Ricci flow \eqref{tkrf} on $M\times[0,\,\infty)$. Then for any $f\in L^1(M)$ and $\gamma>0$, there holds
$$
\gamma \lf|\{x|(Mf)(x)>\gamma\}\rt|\leq C\|f\|_{L^{1}(M)},
$$
where $C$ depends only on $\kappa$ and $n$.
\end{lem}
\begin{proof} The idea comes from \cite[Chapter 3]{fo} (see also \cite{cw,st}).
For any $x\in \{x|(Mf)(x)>\gamma\}=:S_{\gamma}$, there exists $r_x$ such that
$$
\frac{1}{\lf|B(x,\,r_x)\rt|}\int_{B(x,\,r_x)}|f(\zeta)|\dmu(\zeta)>\gamma.
$$
Obviously,
$$
\{B(x,\,r_x)|x\in S_{\gamma}\}
$$
is an open covering of $S_{\gamma}$.
For any $0<c<|S_{\gamma}|$, from measure theory (see for example Theorem 2.40 in \cite{fo}), there exists a compact set $K$ such that $|K|>c$ and finitely many balls, saying $B(x_{1},\,r_{x_{1}}),\cdots,B(x_{p},\,r_{x_{p}})$, cover $K$.  Let $B(x_{i_1},\,r_{x_{i_1}})$ be the ball with the largest radius in $B(x_{i},\,r_{x_{i}})$, let $B(x_{i_2},\,r_{x_{i_2}})$ be the ball with the largest radius in $B(x_{i},\,r_{x_{i}})$'s that are disjoint from $B(x_{i_1},\,r_{x_{i_1}})$, $B(x_{i_3},\,r_{x_{i_3}})$ the ball with the largest radius in $B(x_{i},\,r_{x_{i}})$'s that are disjoint from $B(x_{i_1},\,r_{x_{i_1}})$ and $B(x_{i_2},\,r_{x_{i_2}})$, and so on until the list of $B(x_{i},\,r_{x_{i}})$  is exhausted.
According to the construction above, if $B(x_{i},\,r_{x_{i}})$ is not the one of the  $B(x_{i_j},\,r_{x_{i_j}})$'s, there is a $j$ such that $B(x_{i},\,r_{x_{i}})\cap B(x_{i_j},\,r_{x_{i_j}})\neq\emptyset$, and if $j$ is the smallest integer with this property, the radius of $B(x_{i},\,r_{x_{i}})$ is at most that of $B(x_{i_j},\,r_{x_{i_j}})$. Therefore $B(x_{i},\,r_{x_{i}})\subset B(x_{i_j},\,3r_{x_{i_j}})$ and then
$$
K\subset \cup_{j} B(x_{i_j},\,3r_{x_{i_j}}).
$$
Therefore, from Item (d) and (e) in Proposition \ref{prop1}, we have
\ba
c&<&|K|\leq \sum_{j}|B(x_{i_j},\,3r_{x_{i_j}})|\leq 3^n\kappa^2\sum_{j} |B(x_{i_j},\,r_{x_{i_j}})|\no\\
&\leq &3^n\kappa^2\sum_{j}\int_{B(x_{i_j},\,r_{x_{i_j}})}\frac{|f(\zeta)|\dmu(\zeta)}{\gamma}\leq \frac{3^n\kappa^2}{\gamma}\|f\|_{L^{1}(M)}\no.
\ea
Letting $c\longrightarrow |S_{\gamma}|$, we can deduce the desired conclusion.
\end{proof}
Next, using Lemma \ref{mfunction}, we obtain
\begin{lem}[see \cite{cdg} for the version of the Euclidean space]
\label{lemia}
Let $g(t)$ be the solution to the twisted K\"{a}hler-Ricci flow \eqref{tkrf} on $M\times[0,\,\infty)$. Then for any $f\in L^1(M)$ and $\delta>0$, there holds
\be\label{i4}
\delta^{\frac{n}{n-\al}}\lf|\lf\{x\in M|(I_{\al}f)(x)>\delta\rt\}\rt|\leq C\|f\|_{L^1(M)}^{\frac{n}{n-\al}},\quad 0<\al<n,
\ee
where $C$ depends only on $\kappa,\,n$ and $\al$.
\end{lem}
\begin{proof}
Denote
$$
(I_{\al,1}f)(x):=\int_{B(x,\,\varepsilon)}|f(\zeta)|\frac{[\dist(x,\,\zeta)]^{\al}}{\lf|B(x,\dist(x,\zeta))\rt|}\dmu(\zeta)
$$
and
$$
(I_{\al,2}f)(x):=\int_{M-B(x,\,\varepsilon)}|f(\zeta)|\frac{[\dist(x,\,\zeta)]^{\al}}{\lf|B(x,\dist(x,\zeta))\rt|}\dmu(\zeta).
$$
Then for $0<\varepsilon<\mathrm{diam}(M)$, we have
\be\label{i1}
\bag
(I_{\al,1}f)(x)
=&\sum\limits_{k=0}^{\infty}\int_{\{2^{-(k+1)}\varepsilon\leq \dist(x,\,\zeta)<2^{-k}\varepsilon\}}|f(\zeta)|\frac{[\dist(x,\,\zeta)]^{\al}}{\lf|B(x,\dist(x,\zeta))\rt|}\dmu(\zeta)\\
\leq&\sum\limits_{k=0}^{\infty}(2^{-k}\varepsilon)^{\al}\lf|B(x,2^{-(k+1)}\varepsilon)\rt|^{-1}\int_{B(x,2^{-k}\varepsilon)}|f(\zeta)|\dmu(\zeta)\\
\leq&C\sum\limits_{k=0}^{\infty}(2^{-k}\varepsilon)^{\al}\lf|B(x,2^{-k}\varepsilon)\rt|^{-1}\int_{B(x,2^{-k}\varepsilon)}|f(\zeta)|\dmu(\zeta)\\
\leq&C(Mf)(x)\varepsilon^{\al},
\eag
\ee
where $C$ depends only on $\kappa$ and $n$.

Thus, Lemma \ref{mfunction} and (\ref{i1}) implies
$$
\lf|\lf\{x\in M|(I_{\al,1}f)(x)>\delta\rt\}\rt|
\leq \lf|\lf\{x\in M|(Mf)(x)>\frac{\delta}{C\varepsilon^{\al}}\rt\}\rt|\\
\leq\frac{C\varepsilon^{\al}}{\delta}\|f\|_{L^{1}(M)},
$$
where $C$ depends only on $\kappa$ and $n$.

By Item (d) and (e) in Proposition \ref{prop1}, we derive
$$
\sup_{\zeta\in M-B(x,\,\varepsilon)}\frac{[\dist(x,\,\zeta)]^{\al}}{\lf|B(x,\,\dist(x,\,\zeta))\rt|}\leq \kappa \sup_{\zeta\in M-B(x,\,\varepsilon)}[\dist(x,\,\zeta)]^{\al-n}
=\kappa \varepsilon^{\al-n},
$$
which implies
\be\label{i2}
(I_{\al,2}f)(x)\leq \kappa \varepsilon^{\al-n}\|f\|_{L^1(M)}.
\ee
From (\ref{i1}) and (\ref{i2}), taking
$$
\varepsilon=\lf(\frac{\|f\|_{L^1(M)}}{(Mf)(x)}\rt)^{\frac{1}{n}},
$$
we have
\be\label{i3}
(I_{\al}f)(x)\leq C\lf[(Mf)(x)\rt]^{\frac{n-\al}{n}}\|f\|_{L^1(M)}^{\frac{\al}{n}},
\ee
where $C$ depends only on $\kappa$ and $n$. We remark that $(I_{\al,2}f)(x)=0$ if $\varepsilon>\mathrm{diam}(M)$.

Combining Lemma \ref{mfunction} and (\ref{i3}), we arrive at
$$
\bag
\lf|\lf\{x\in M|(I_{\al}f)(x)>\mu\rt\}\rt|
\leq &\lf|\lf\{x\in M|(Mf)(x)>\mu^{\frac{n}{n-\al}}C^{-\frac{n}{n-\al}}\|f\|_{L^1(M)}^{-\frac{\al}{n-\al}}\rt\}\rt|\\
\leq&\mu^{-\frac{n}{n-\al}}C^{\frac{n}{n-\al}}\|f\|_{L^1(M)}^{\frac{\al}{n-\al}}C\|f\|_{L^1(M)},
\eag
$$
which implies (\ref{i4}).
\end{proof}
This lemma together with Lemma \ref{lemg} implies
\begin{lem}\label{lemw}
Let $g(t)$ be the solution to the twisted K\"{a}hler-Ricci flow \eqref{tkrf} on $M\times[0,\,\infty)$. Then for any $\mu>0$ and $f\in C^{\infty}(M,\R)$ with $\intm f\dmu=0$, there holds
\be\label{i5}
\delta^{\frac{n}{n-1}}\lf|\lf\{x\in M||f(x)|>\delta\rt\}\rt|\leq C\|\n f\|_{L^1(M)}^{\frac{n}{n-1}},
\ee
where $C$ is a uniform constant.
\end{lem}
\begin{proof}
Since $\Delta f=\Delta f$ and $\intm f\dmu=0$, by integration by parts, we have
\be\label{f1}
\bag
f(x)
=&-\intm G_0(x,\,z)\Delta f(z)\dmu(z)\\
=&-\lim_{r\longrightarrow 0}\int_{M-B(x,\,r)}G_0(x,\,z)\Delta f(z)\dmu(z) \\
=& \lim_{r\longrightarrow 0}\int_{M-B(x,\,r)}\n_{z}G_0(x,\,z)\n f(z)\dmu(z)- \lim_{r\longrightarrow 0}\int_{\de B(x,\,r)}G_0(x,\,z)\de_{\overrightarrow{n}}f(z)\md S(z) \\
=&\intm \n_{z}G_0(x,\,z)\n f(z)\dmu(z),
\eag
\ee
where $\overrightarrow{n}$ is the inward normal vector on $B(x,\,r)$, and we use Item (d) and (e) in Proposition \ref{prop1} and the fact
$$
|G_0(x,z)|\leq \frac{C}{[\dist(x,\,z)]^{n-2}}.
$$
Using Lemma \ref{lemg}, (\ref{f1}), and Item (d) and (e) in Proposition \ref{prop1} together, we have
\be\label{f2}
\bag
|f(x)|
\leq& C\intm  \frac{|\n f(z)|}{[\dist(x,\,z)]^{n-1}}\dmu(z)\\
\leq& C\intm   |\n f(z)|\frac{\dist(x,\,z)}{\lf|B(x,\,\dist(x,z))\rt|}\dmu(z)\\
=&C(I_1|\n f|)(x),
\eag
\ee
where $C$ is a uniform constant.
Thus, the lemma follows from Lemma \ref{lemia} and (\ref{f2}).
\end{proof}
Using Lemma \ref{lemw} and the arguments in \cite{tianzhang}, we can also deduce
\begin{lem}
\label{lemsos}
Let $g(t)$ be the solution to the twisted K\"{a}hler-Ricci flow \eqref{tkrf} on $M\times[0,\,\infty)$. Then for any $f\in C^{\infty}(M,\R)$, there holds
\be\label{wsobstar}
\|f\|_{L^{\frac{n}{n-1}}(M)}\leq C\|\n f\|_{L^1(M)}+C_1\lf[\mathrm{Vol}(M)\rt]^{-\frac{1}{n}}\|f\|_{L^1(M)},
\ee
where
$C$ is a uniform constant.
\end{lem}
Now we can give the proof of our main theorem as follows.
\begin{proof}[Proof of Theorem \ref{thmiso}]
For any $f\in C^{\infty}(M,\,\R)$, from (\ref{wsobstar}), we have
\be\label{isoyong2}
\|f-f_M\|_{L^{\frac{n}{n-1}}(M)}\leq C\|\n f\|_{L^1(M)}+C_1\lf[\mathrm{Vol}(M)\rt]^{-\frac{1}{n}}\|f-f_M\|_{L^1(M)}.
\ee
By using (\ref{f2}), we get
\be\label{isoyong1}
|f(x)-f_M|
\leq C\intm  \frac{|\n f(z)|}{[\dist(x,\,z)]^{n-1}}\dmu(z).
\ee
Applying the discussion in (\ref{i1}), (\ref{isoyong1}), and Item (d) and (e) in Proposition \ref{prop1},  we can deduce
$$
\|f(x)-f_M\|_{L^1(M)}\leq C\mathrm{diam}(M)\|\n f\|_{L^1(M)},
$$
which, by (\ref{isoyong2}) added, implies isoperimetric inequality
$$
\|f(x)-f_M\|_{L^{\frac{n}{n-1}}(M)}\leq C\|\n f\|_{L^1(M)},
$$
where
$C$ is a uniform constant.
\end{proof}

%\bibliographystyle{abbrv}
%\bibliography{D:/bib/refs_General}
\end{CJK}
\end{document}